\long\def\skipit#1{} 
\def\={\,=\,}
\def\+{\,+\,}
\def\-{\,-\,}

\documentclass[11pt]{amsart}

\renewcommand{\today}{\number\year-0\number\month-\number\day  }

\usepackage{pstricks-add}
\usepackage{amsfonts,amssymb,amsmath,amsthm,mathrsfs}
\usepackage{ifthen,calc}
\usepackage{color}
\usepackage{epsfig}
\usepackage{graphicx}
\usepackage{epstopdf}
\usepackage{epsfig}
\usepackage{latexsym}
\usepackage{multicol}
\usepackage[format=hang, margin=10pt]{caption}
\usepackage[mathscr]{eucal}
\usepackage{graphicx}
\usepackage{epstopdf}
\usepackage{epsfig}
\usepackage{float}

\newcounter{hours}
\newcounter{minutes}
\newcommand{\printtime}{
	\setcounter{hours}{\time/60}%
	\setcounter{minutes}{\time-\value{hours}*60}
	\ifthenelse{\value{hours}<10}{0}{}\thehours:%
	\ifthenelse{\value{minutes}<10}{0}{}\theminutes}

\numberwithin{equation}{section}
\numberwithin{figure}{section}
\numberwithin{table}{section}

\newtheorem{thm}{Theorem}[section]

\newtheorem{lemma}[thm]{Lemma}
\newtheorem{prop}[thm]{Proposition}

\newtheorem{J-com}{JG-comment}[section]
\theoremstyle{definition}
\newtheorem{example}{Example}[section]

\newtheorem{defn}{Definition}[section]

\newtheorem{property}{Property}[section]
\newtheorem{rem}[thm]{Remark}

\keywords{hypergraph, hypermap, hypertree, partial duality, partial-dual Euler-genus polynomial}
\subjclass{Primary: 05C10}

\begin{document}

\title[Enumerating Partial Duals of Hypermaps by Genus]{Enumerating Partial Duals of Hypermaps by Genus}

\author{Wenwen Liu}
\address{School of Mathematics and Physics, SuZhou University of Science and Technolgy, 215009 SuZhou, China}
\email{1991807694@qq.com}
\author{Yichao Chen}
\address{School of Mathematics and Physics, SuZhou University of Science and Technolgy, 215009 SuZhou, China}
\email{ycchen@hnu.edu.cn}

\begin{abstract} The concept of partial duality in hypermaps was introduced by Chmutov and Vignes-Tourneret, and  Smith independently. This notion serves as a generalization of the concept of partial duality found in maps.
In this paper, we first present an Euler-genus formula concerning the partial duality of hypermaps, which serves as an invariant related to the result obtained by Chmutov and Vignes-Tourneret. This formulation also generalizes the result of Gross, Mansour, and Tucker regarding partial duality in maps. Subsequently, we enumerate the distribution of partial dual Euler-genus for hypermaps and compute the corresponding polynomial for specific classes of hypermaps through three operations: join, bar-amalgamation, and subdivision.

 \end{abstract}

\maketitle

\bigskip
\section{Introduction}  


\subsection{ Background}

{Let  $H=(V(H),E(H))$ denote a \textit{hypergraph}, where $V(H)$  represents the vertex set and $E(H)$ denotes the hyperedge set. Each hyperedge is defined as a nonempty subset of vertices. If $e\in E(H)$, and if $v_{1},v_{2},\ldots,v_{i} \in e$, we say that the vertices $v_{1},v_{2},\ldots,v_{i}$ are \textit{adjacent} to each other. Furthermore, we state that the vertex $v_{1}$ (or any of the vertices $v_{2},v_{3},\ldots,v_{i}$) is incident with the hyperedge $e$. Given two distinct hyperedges  $e_{1},e_{2}$, if $e_{1}\cap e_{2}\neq\emptyset$, we say that $e_{1}$ and $e_{2}$  are adjacent to each other.  The \textit{degree of a vertex} $v_{i}\in V(H)$, denoted as  $d_{i}$, is defined as the number of hyperedges with which $v_{i}$ is adjacent.    The \textit{degree of a hyperedge} $e_{i}\in E(H)$, donated as $n_{i}$,  refers to the number of vertices contained within that hyperedge.   For all hyperedges  $e\in E(H)$, if the degree of an edge equals $r(r\geq 2)$, then the hypergraph is referred to as an $r$-\textit{uniform hypergraph}. It is evident that a graph qualifies as a $2$-uniform hypergraph.}

A \textit{chain of length $q$ in a hypergraph $H$ }is a sequence $v_{1}e_{1}v_{2}e_{2}v_{3}\cdots v_{q}e_{q}v_{q+1}$, where $v_{1},v_{2},\ldots,$ $v_{q},v_{q+1}\in V(H)$, $e_{1},e_{2},\ldots,e_{q}\in E(H),$ such that

(1) All the vertices $v_{1},v_{2},\ldots,v_{q},v_{q+1}$, with the exception of  $v_{1}$ and $v_{q+1}$, are distinct;

(2) All the hyperedges $e_{1},e_{2},\ldots,e_{q}$ are distinct;

(3) $v_{i}, v_{i+1} \in e_{i}$ for all $i = 1, 2, 3, \ldots, q$.


{If $v_{1} \neq v_{q+1}$, the chain $v_{1}e_{1}v_{2}e_{2}v_{3}\cdots v_{q}e_{q}v_{q+1}$ is referred to as a \textit{path of length $q$}; if $q \geq 2$ and $v_1 = v_{q+1}$, the chain is termed a \textit{cycle of length $q$ }. Therefore, a hypergraph is considered connected if there exists a path connecting any two vertices within the hypergraph $H$.}

{Every connected hypergraph is associated with a connected bipartite graph, as established by Walsh {\cite{Wa75}}. For a given connected hypergraph $H$, we can construct the corresponding bipartite graph $G_H$ such that $V(G_H)=V(H)\cup E(H)$ and $E(G_H)=\{\{v,e\}|v\in V(H),e\in E(H)\; and\;v\in e \}$, as illustrated in Figure \ref{p1}.}

\begin{figure*}[ht]
    \centering
    \includegraphics[width=4in]{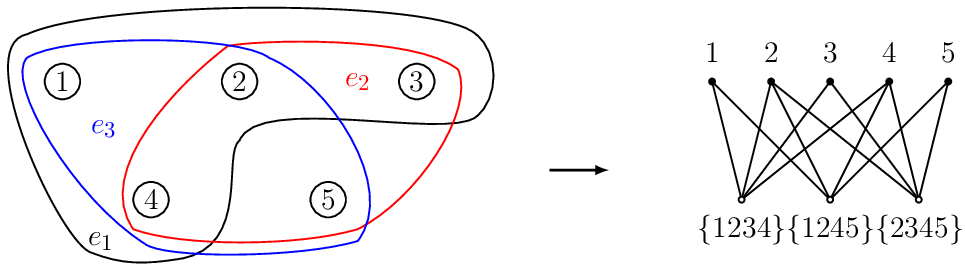}

\caption{A hypergraph $H$ and its associated bipartite graph $G_H$.}
    \label{p1}
\end{figure*}


{A (closed) \textit{surface} in this paper refers to a compact 2-manifold without boundary, denoted as $S_{n}$, where $n$ indicates the number of handles or crosscaps present in the surface. This value represents the genus of the surface. Additionally, a hypermap can be understood as a hypergraph that is embedded on a surface, serving as an extension of the concept of a map (or a graph embedding). In this paper, we primarily concentrate on the concepts of partial duality and Euler-genus polynomials related to hypermaps.}

{In 2009, Chmutov generalized the geometric duality of maps and introduced the concept of partial duality for maps in \cite{C09}. For any subset $A$ of edges of a map $M$, a partial dual $M^{A}$  represents a geometric duality of $M$ with respect to $A$. In 2020, Gross, Mansour and Tucker introduced the partial-dual Euler-genus polynomial for maps in {\cite{GMT20}}, demonstrating that the partial-dual genus polynomials for all orientable maps are interpolating.}

{Recently, Chmutov and Vignes-Tourneret  {\cite{CV22}} and Simith \cite{Smi18} independently   defined the notion of partial duality for hypermaps. In  {\cite{CV22}}, Chmutov and Vignes-Tourneret provided a formula to describe the change in genus under partial duality. The authors have advocated for a deeper investigation into the polynomials related to the partial duality of hypermaps. The objective of this paper is to address this matter.}

{This paper is organized as follows. In Section 2, we use a combinatorial model developed by Tutte to define the concept of partial duality in hypermaps and introduce the partial-dual Euler-genus polynomial associated with these structures. Sections 3 and 4 present three operations: \textit{join}, \textit{bar-amalgamation}, and \textit{subdivision}, along with their corresponding Euler-genus polynomials for partial-duals. Finally, we explore the properties of hypertree maps and illustrate these concepts through relevant examples.}

\subsection{{Bi-rotation system}}
Here we introduce Tutte's permutation axiomatization for maps \cite{Tut84}.  A \textit{rotation at a vertex} $v$ of a graph $G$ refers to the cyclic ordering of the edge-ends incident at $v$. A \textit{rotation system} $R$ of a graph $G$ is defined as an assignment of a rotation at every vertex of $G.$ An embedding of a graph $G$ on an arbitrary surface $S$ can be described combinatorially by a \textit{signed rotation system} $(R,\sigma)$. Here, $R$ represents the rotation system of $G$, and $\sigma$ is the \textit{twist-indicator} \cite{CG18}.  Specifically, if $\sigma(e)=-1$, then the edge $e$ is twisted; otherwise, if $\sigma(e)=1$, then the edge $e$ is untwisted. For each edge \( e_{i}\) in \( G \) is labeled with four integers: \( 4i-3, 4i-2, 4i-1, 4i \). Here, the labels \( 4i-3 \) and \( 4i-2 \) are associated with the neighborhood of vertex \( u \), while labels \( 4i-1 \) and \( 4i \) correspond to the neighborhood of vertex \( v\), where vertices \( u \) and \( v\) are the endpoints of edge \( e_{i} \). The labels on the left side of edge \( e_{i} \) are denoted as \( 4i-3\) and \( 4i-1\), whereas those on the opposite side are labeled as $4i-2$ and $4i$. Consequently, if the edge \( e_{i} \) is untwisted, then we let $e_{i}=(4i-3,4i-1)(4i-2,4i)$, otherwise we let $e_{i}=(4i-3,4i)(4i-2,4i-1).$  Let \( B \) represent the set of all labels associated with the edges in graph \( G \). It is evident that the cardinality of \( B \) is equal to \( 4e(G) \), where \( e(G) \) denotes the number of edges in \( G \). The product of the cyclic permutations corresponding to all edges is denoted as
\(
\psi =\prod\limits_{i=1}^{e(G)}e_i.
\)


For every vertex \( u \) in the graph \( G \), the number of labels in the neighborhood of \( u \) is given by \( 2d(u) \). Let the cyclic permutation \( R^{l} \) consist of all labels on the left side of an edge arranged in a clockwise rotation, while \( R^{r} \) comprises all right-side labels arranged in a counterclockwise rotation. The \textit{bi-rotation} associated with vertex \( u \) is defined as \( R^{l} \circ R^{r} \).

The union of all bi-rotations for each vertex forms what we refer to as the \textit{bi-rotation system} of graph \( G\), denoted by \( \tau\). Consequently, an embedding of graph \( G\), represented as a map \( M\), can be viewed as a triple \( M = (B, \tau, \psi)\),
where the composition \( \psi\circ\tau\) represents the product of all bi-rotations corresponding to each face within embedding \( M\).


\begin{defn}
Given a map \( M = (B, \tau, \psi) \), let \( D \) be a nonempty subset of vertices. The restriction of \( \tau \) to the set \( D \), denoted as \( \tau|_{D} \), is defined by removing all labels in \( \tau \) that are not contained within \( D\).
\end{defn}

\subsection{{The bipartite graph model for hypermaps}}  

{A hypergraph can be represented as a bipartite graph; therefore, we can define a hypermap through an embedding of the corresponding bipartite graph.}
\begin{defn}
\label{def2}Let \( M_{H} = (B, \tau, \gamma) \) be an embedding of the corresponding bipartite graph, and \( D = V(H) \). A hypermap can be defined as follows:

(1) \( \tau(V(H)) = \tau|_{D} \);

(2) \( (\psi\circ\tau)|_{D} \) represents the product of all bi-rotations of the faces of \( H \);

(3) \( \psi(E(H)) = (\psi\circ\tau)|_{D} \circ (\tau|_{D})^{-1} \).\\
Similarly, a hypermap can be also viewed as a triple $\mathcal{H}=(B|_{D},\tau(V(H)),\psi(E(H)))$
\begin{rem}
In summary, provided that no contradictions arise from this process, a hypermap may still be denoted as \( \mathcal{H} = (B, \tau, \psi) \).

Next, there is no distinction between hyergraph and hypermap in this paper, and their symbols are uniformly denoted as $H$.
\end{rem}
\end{defn}
Figure \ref{p2} and Figure \ref{p3} illustrate the processes of obtaining a plane embedding and a toroidal embedding of the hypergraph $H$ shown in Figure \ref{p1}, using the aforementioned rules.

\begin{example}
\begin{figure}[H]
    \centering
    \includegraphics[width=5in]{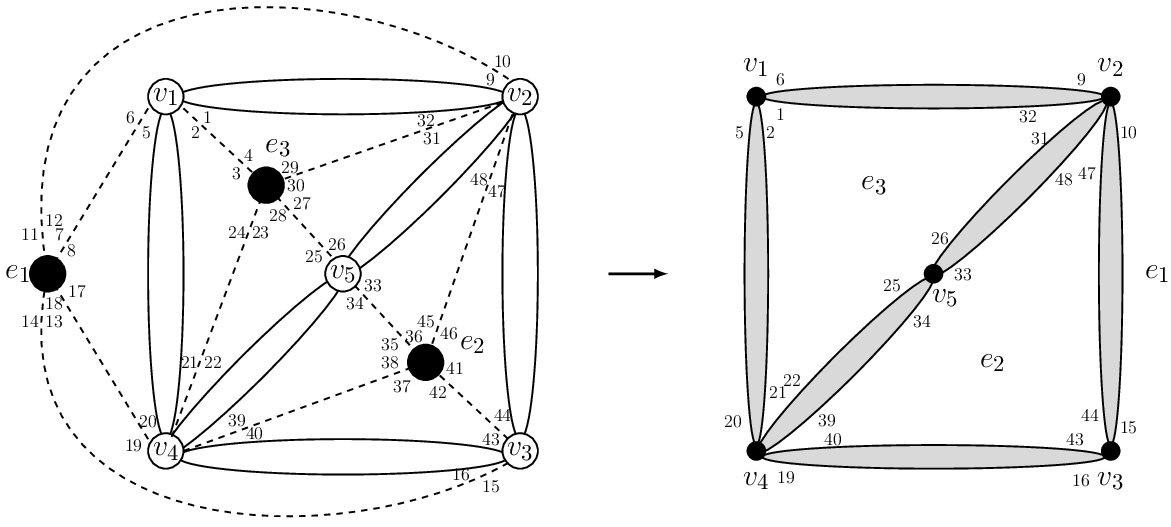}
    \caption{A bi-rotation system of a planar hypermap.}
    \label{p2}
\end{figure}
In Figure $\ref{p2}$, we define the bipartite map $M_{H}=(B,\tau,\psi)$, where
\begin{align*}
B&=\{1,2,3,\ldots,45,46,47,48\},\\
\tau&=(1,5)(2,6)(9,47,31)(10,32,48)(15,43)(16,44)(19,21,39)(20,40,22)(25,33)(26,34)\\ &\ \ \ \ (3,29,27,23)(4,24,28,30)(35,45,41,37)(36,38,42,46)(7,17,13,11)(8,12,14,18).\\
\psi&=(1,3)(2,4)(5,7)(6,8)(9,11)(10,12)(13,15)(14,16)(17,19)(18,20)(21,23)(22,24)(25,27)\\ &\ \ \ \ (26,28)(29,31)(30,32)(33,35)(34,36)(37,39)(38,40)(41,43)(42,44)(45,47)(46,48).
\end{align*}

By Definition $\ref{def2}$, we can construct a hypermap $H=(B|_{D},\tau(V(H)),\psi(E(H)))$ in the following manner. Here
\begin{align*}
B|_{D}&=\{1,2,5,6,9,10,15,16,19,20,21,22,25,26,31,32,33,34,39,40,43,44,47,48\},\\
\tau(V(H))&=(1,5)(2,6)(9,47,31)(10,32,48)(15,43)(16,44)(19,21,39)(20,40,22)(25,33)(26,34),\\
\psi(E(H))&=(1,31,25,21)(2,22,26,32)(5,19,15,9)(6,10,16,20)(33,47,43,39)(34,40,44,48).
\end{align*}

\end{example}

\begin{figure}[H]
    \centering
    \includegraphics[width=2.6in]{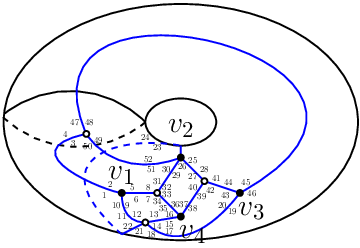}
    \includegraphics[width=2.6in]{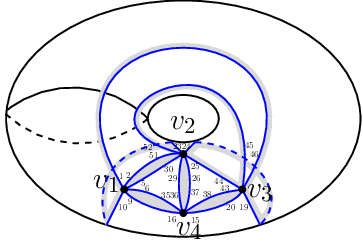}
    \caption{A hypermap on torus.}
    \label{p3}
\end{figure}

\begin{example}
In Figure $\ref{p3}$, we present an embedding of a bipartite graph on a torus, denoted as $M_{H}=(B,\tau,\psi)$, where
\begin{align*}
B&=\{1,2,3,\ldots,48,49,50,51,52\},\\
\tau&=(1,5,9)(2,10,6)(3,47,49)(4,50,48)(7,31,33)(8,34,32)(11,13,17,21)(12,22,18,14)\\& \ \ \ (15,35,37)(16,38,36)(19,43,45)(20,46,44)(23,25,29,51)(24,52,30,26)(27,41,39)\\ & \ \ \ (28,40,42),\\
\psi&=(1,3)(2,4)(5,7)(6,8)(9,11)(10,12)(13,15)(14,16)(17,19)(18,20)(21,23)(22,24)\\
&\ \ \ (25,27)(26,28)(29,31)(30,32)(33,35)(34,36)(37,39)(38,40)(41,43)(42,44)(45,47)\\ & \ \ \ (46,48)(49,51)(50,52).
\end{align*}Then, a hypermap \( H = (B|_{D}, \tau(V(H)), \psi(E(H))) \) is derived from the bipartite map \( M_{H} = (B, \tau, \psi) \) as follows. Here
\begin{align*}
 B|_{D}&=\{1,2,5,6,9,10,15,16,19,20,23,24,25,26,29,30,35,36,37,38,43,44,45,46,51,52\}\\
\tau(V(H))&=(1,5,9)(2,10,6)(15,35,37)(16,38,36)(19,43,45)(20,46,44)(23,25,29,51)\\ &\ \ \ (24,52,30,26)\\
\psi(E(H))&=(1,45,51)(2,52,46)(5,29,35)(6,36,30)(9,15,19,23)(10,24,20,16)(25,43,37)\\ &\ \ \ (26,38,44)
\end{align*}
\end{example}
\subsection{The arrow presentation of a hypermap $H$}
A hypermap can be represented as a ribbon graph. Specifically, we can replace the vertices, hyperedges, and faces with vertex-disks, hyperedge-disks, and face-disks such that disks of the same type do not intersect, while disks of different types may only intersect at boundary arcs.

\begin{figure}[H]
    \centering

        \includegraphics[width=2in]{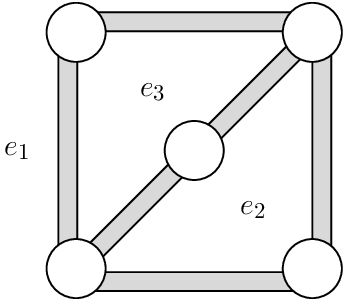}
        \includegraphics[width=2.45in]{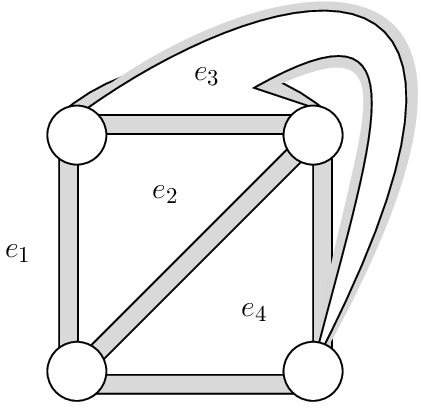}

    \caption{ Hypermaps as ribbon graphs.}
    \label{p5}
\end{figure}

\begin{figure}[H]
    \centering
        \includegraphics[width=2.6in]{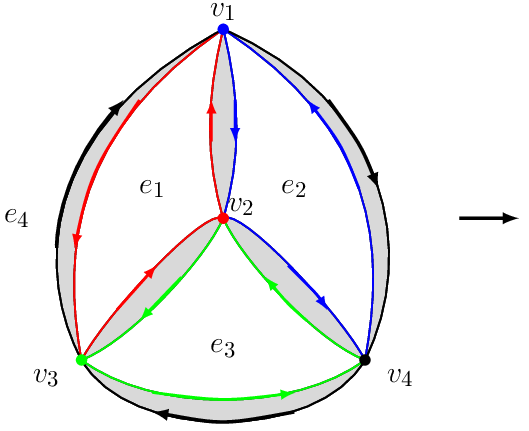}
        \includegraphics[width=2in]{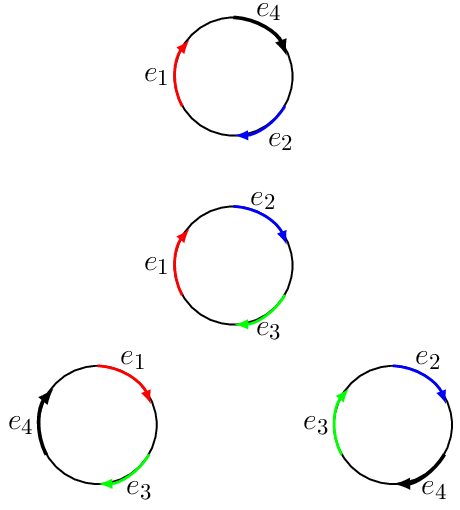}
    \caption{The arrow presentation of a hypermap.}
    \label{p6}
\end{figure}
The \textit{arrow presentation} of a hypermap $H$, denoted $\vec{H}$, consists of a set of the vertex-disks in $H$ with labelled arrows, called \textit{marking arrows}.  These marking arrows represent the common segments on the vertex-disks and are labeled according to the hyperedges they intersect. The direction of these arrows aligns with the boundaries of their corresponding hyperedges.  Obviously, the number of marking arrows labelled $e_{i}$ is $n_{i}$. This process is illustrated in Figure \ref{p6}.


\subsection{The Euler-characteristic of a hypermap}
\begin{defn}
\label{def3}
Given a hypermap $H = (B, \tau, \psi)$, let $v(H),e(H), f(H)$ represent the number of vertices, hyperedges and faces of $H$, respectively. For each hyperedge $e_{i}\in E(H)$, let $n_i$ be the number of vertices incident to $e_{i}$, then the Euler-characteristic of the hypermap $H$ is given by
$$\chi(H)=v(H)+e(H)+f(H)-\sum\limits_{i=1}^{e(H)}n_{i}.$$
\end{defn}

\begin{rem}
We provide a comprehensive analysis of the previously mentioned definition
The terms $v(G_H)$, $e(G_H)$, and $f(G_H)$ denote the number of vertices, edges, and faces in the bipartite graph $G_H$, respectively. By definition \ref{def2}, we have $\chi(H) = \chi(G_H)$ and $f(H) = f(G_H)$. Letting $n_{i}$ represent the number of vertices incident to edge $e_{i} \in E(H)$ (the number of vertices contained in edge $e_{i}$), we find that \(E(G_H)=\{\{v,e\}|v\in V(H), e\in E(H), v\in e\}\). Consequently, \(e(G_H)=\sum\limits_{i=1}^{e(H)} n_{i}\). Thus, $\chi(H) = \chi(G_H) = v(G_H) + f(G_H) - e(G_H) = v(H) + e(H) + f(H) - \sum\limits_{i=1}^{e(H)} n_{i}.$
\end{rem}

 \section{The partial duality for a hypermap}
In this paper, we focus exclusively on the partial duality of hyperedges. However, this definition can be easily extended to encompass vertex and face partial duality due to the inherent symmetry of hypermaps, as discussed in \cite{{CV22}}.

%
%

Given a subset of hyperedges $A\subseteq E(H)$, we consistently regard $A$ as a spanning sub-hypermap, which we will continue to denote as  $A$. Consequently, it follows that $v(A)=v(A^{c})=v(H)$, where $A^{c}=E(H)-A$.
%
%
%
%
%
%
%
%

\begin{defn}
Given a hypermap \( H = (B, \tau, \psi) \), for any subset of hyperedges \( A \), the partial-dual $H^{A}$ of the hypermap \( H \) with respect to \( A \) is defined as follows:

(1) $\tau(H^{A})=\psi|_{A}\circ\tau$;

(2) $\psi(H^{A})=\psi$.
\end{defn}
Obviously, according to the definition of partial duality, we can derive the following properties.
\begin{property}\cite{CV22}
Given a connected hypermap \( H \) and a subset of hyperedges \( A \subseteq E(H) \), the partial dual of \( H \) with respect to \( A \) exhibits the following properties:
\[
(1)\;v(H^{A})=f(A);\indent\indent (2)\;e(H^{A})=e(H);\indent\indent (3)\;f(H^{A})=v((H^{A})^{*}).
\]
\end{property}
\begin{proof}

The expression \( v(H^{A}) = \frac{\|\tau(H^{A})\|}{2} = f(A) \) holds true, as the composition \( \psi|_{A} \circ \tau \) represents the product of all bi-rotations across all faces on \( A \). Here, \( \|\circ\| \) denotes the number of orbits within the permutation.

In particular, when \( A = E(H) \), it follows that \( \tau(H^{A}) = \psi \circ\tau\). Consequently, we have \( f(H^{A}) = v([H^{A}]^{*})\).
\end{proof}

\begin{example}

In Figure \ref{p7}, let \( A = \{e_{1}\} \).   Given that
$$\psi=(1,5,19)(4,18,8)(3,11,21)(2,24,10)(7,9,15)(6,14,12)(13,23,17)(16,20,22)$$ and
$$\tau=(1,17,21)(2,22,18)(7,13,19)(8,20,14)(3,9,5)(4,6,10)(11,23,15)(12,16,24),$$ we find that
$$\psi|_{A}=(1,5,19)(4,18,8)$$ and
$$\tau(H^{A})=\psi|_{A}\circ\tau=(1,3,9,5,7,13,19,17,21)(2,22,18,20,14,8,6,10,4)(11,23,15)(12,16,24).$$
\begin{figure}[H]
    \centering
    \includegraphics[width=4in]{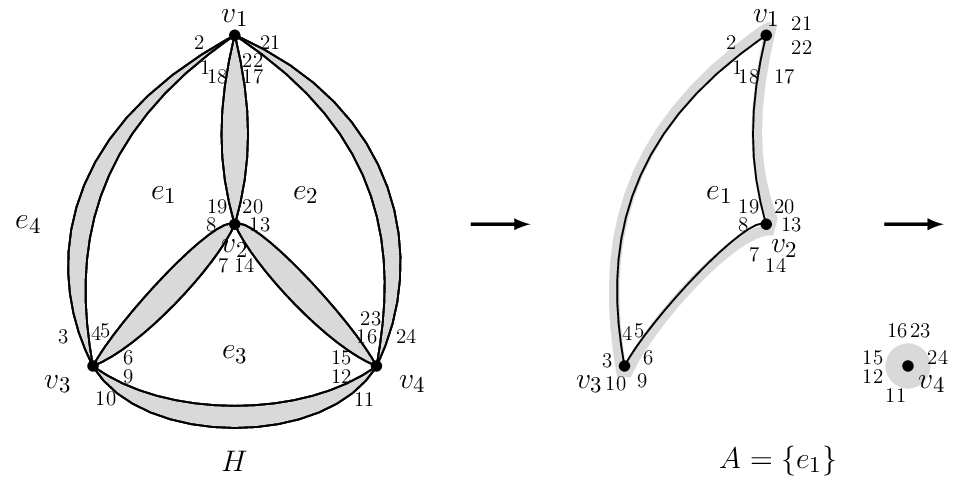}
    \includegraphics[width=2in]{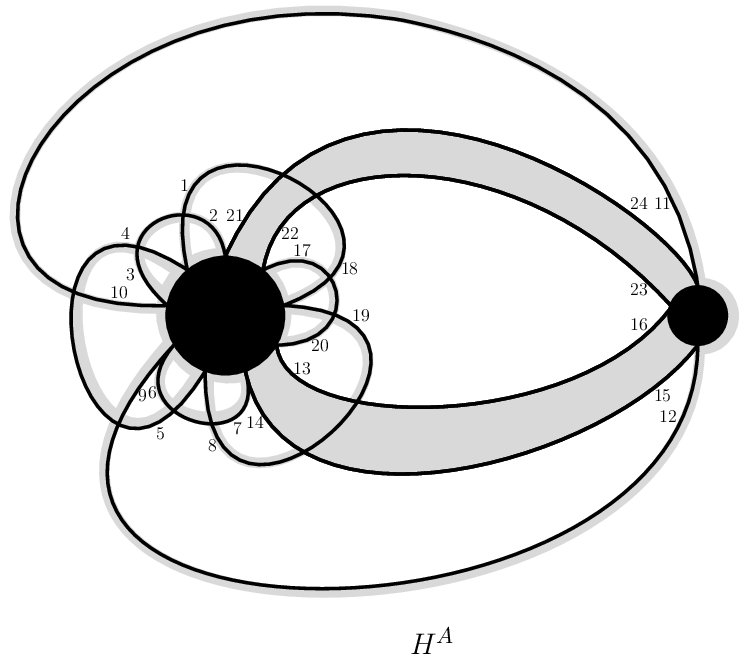}
    \caption{The partial dual of the hypermap $ H $ with respect to $ A$, where $A=\{e_{1}\}$.}
    \label{p7}
\end{figure}
\end{example}

\begin{figure}[H]
    \centering
    \includegraphics[width=2in]{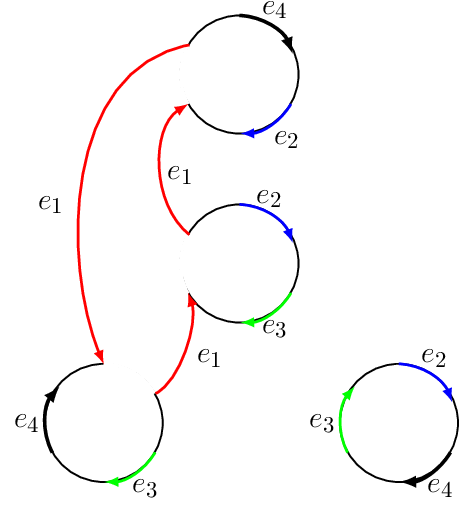}
    \setlength{\belowcaptionskip}{-0.4cm}
    \includegraphics[width=3in]{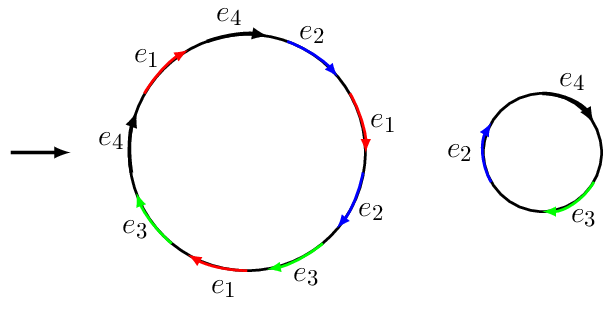}
    \caption{The arrow presentation of the partial dual $H^{A}$ of the hypermap in Figure \ref{p6}, where $A=\{e_{1}\}$.}
    \label{p8}

\end{figure}

Let $\vec{H}$ denote the arrow presentation of a hypermap $H$, and let $A$ be a subset of $E(H)$. In this context, there are $n_{i}$ marking arrows in $\vec{H}$, labeled as $e_{i} \; (\forall e_{i} \in A)$, which we refer to as $\alpha_{1}, \alpha_{2}, \ldots, \alpha_{n_{i}}$.

To illustrate the process, draw a directed line segment with an arrow from the head of $\alpha_{j}$ to the tail of $\alpha_{j+1}$ for each \( j = 1, 2, \ldots, n_{i}-1\). Additionally, create an arrow segment directed from the head of $\alpha_{n_i}$ to the tail of $\alpha_1$. Finally, label these newly created arrows with \( e_i \) and remove both the \( n_i \) marking arrows and their corresponding arcs. The resulting structure is referred to as the partial dual of \( H \) with respect to \( A\). This procedure is illustrated in Figure \ref{p8}.

\begin{property}\label{property1} \cite{CV22}
   Given a connected hypermap \( H \) and a subset \( A \subseteq E(H) \), the following properties hold:

(1) The number of components satisfies \( c(H^{A}) = c(H) \), and the sum of the incidences is given by \( \sum\limits_{i=1}^{e(H^{A})} n_{i}(H^{A}) = \sum\limits_{i=1}^{e(H)} n_{i}(H) \);

(2) If \( H \) is orientable, then \( H^{A} \) is also orientable;

(3) If \( A = E(H) \), then it follows that \( H^{A} = H^{*} \);

(4) For subsets \( A, B \subseteq E(H) \), we have the equality:
\[(H^{A})^{B} = (H^{B})^{A} = H^{A\cap B - A\cup B},\]
and additionally,
\[(H^{A})^{A\cap B - A\cup B} = H^B;\]

(5) The duality relation holds as follows:
\( (H^A)^* = (H^*)^A = H^{A^c}; \)

(6) Finally, it can be stated that
\( (H^A)^A = H. \)
\end{property}

We present the following theorem, which serves as an invariant for a formula regarding the genus change under partial duality, as established in \cite{CV22}. This theorem represents a generalization of the results obtained in \cite{GMT20}.

\begin{thm}\label{th2}
Given a connected hypermap \( H \) and a subset \( A \subseteq E(H) \), the Euler characteristic of the partial dual hypermap \( H^{A} \) is given by
\[
\chi(H^{A}) = \chi(A) + \chi(A^{c}) - 2v(H).
\]
\end{thm}
\begin{proof}
 Due to the equation $\chi(H) = v(H) + e(H) + f(H) - \sum\limits_{i=1}^{e(H)} n_{i}(H)$, along with the conditions $v(A) = v(A^{c}) = v(H)$ and (1), (5) of Property \ref{property1}, we can conclude that
\begin{equation}
\begin{aligned}
\chi(H^{A})&=v(H^{A})+e(H^{A})+f(H^{A})-\sum\limits_{i=1}^{e(H^{A})}n_{i}(H^{A})\\
           &=f(A)+e(H)+v((H^{A})^{*})-\sum\limits_{i=1}^{e(H)}n_{i}(H)\\
           &=f(A)+\left[e(A)+e(A^{c})\right]+v(H^{A^{c}})-\sum\limits_{i=1}^{e(H)}n_{i}(H)\\
           &=f(A)+\left[e(A)+e(A^{c})\right]+f(A^{c})-\sum\limits_{i=1}^{e(H)}n_{i}(H)\\
           &=\left[f(A)+e(A)-\sum\limits_{i=1}^{e(A)}n_{i}(A)\right]+\left[f(A^{c})+e(A^{c})-\sum\limits_{i=1}^{e(A^{c})}n_{i}(A^{c})\right]\\
           &=\chi(A)-v(A)+\chi(A^{c})-v(A^{c})\\
           &=\chi(A)+\chi(A^{c})-2v(H)\nonumber
\end{aligned}
\end{equation}
\end{proof}

\begin{thm}\label{th3}
Let \( H \) be a connected hypermap, and let \( A \subseteq E(H) \). Then, we have the following equations:
\begin{equation}
    \begin{aligned}
        \varepsilon(H^{A}) &= \varepsilon(A) + \varepsilon(A^{c}) + 2[c(H) - c(A) - c(A^{c})] + 2v(H);\\
        \gamma(H^{A}) &= \gamma(A) + \gamma(A^{c}) + c(H) - c(A) - c(A^{c}) + v(H)\nonumber.
    \end{aligned}
\end{equation}
Here, \( \varepsilon(H^{A}) \) and \( \gamma(H^{A}) \) denote the partial-dual Euler-genus and the partial-dual orientable genus of the hypermap \( H^A\), respectively.
\end{thm}
\begin{proof}
According to the equations $\varepsilon(H) = 2c(H) - \chi(H)$, $c(H^{A}) = c(H)$, and $\chi(H^{A}) = \chi(A) + \chi(A^{c}) - 2v(H)$, we can derive the following:

\begin{equation}
    \begin{aligned}
    \varepsilon(H^{A}) & = 2c(H^{A}) - \chi(H^{A}) \\
                       & = 2c(H) - [\chi(A) + \chi(A^{c}) - 2v(H)] \\
                       & = [2c(A) - \chi(A)] + [2c(A^{c}) - \chi(A^{c})] + [2c(H) - 2c(A) - 2c(A^{c})] + 2v(H) \\
                       & = \varepsilon(A) + \varepsilon(A^{c}) + 2[c(H)- c(A)- c(A^{c})] + 2v(H).\nonumber
    \end{aligned}
\end{equation}

If $H$ is an oriented hypermap, then $H^A$, $A$, and $A^C$ are all orientable. Therefore, we have:

\begin{equation}
    \gamma (H^A)=\frac{1}{2}\varepsilon (H^A)=\gamma (A)+\gamma ( A ^ { c })+ c( H )- c( A )- c( A ^ { c })+ v( H ).\nonumber
\end{equation}
\end{proof}

\subsection{The partial-dual genus polynomial}
\begin{defn}\label{def5}
Let \( H \) be a connected hypermap. The partial-dual Euler-genus polynomial of \( H \) serves as the generating function that enumerates partial-duals according to their Euler genus:

\[
\partial_{\varepsilon_{H}}(z)=\sum\limits_{A\subseteq E(H)}z^{\varepsilon(H^{A})}.
\]

In a similar manner, the partial-dual (orientable) genus polynomial of \( H \) is defined as the generating function that enumerates partial duals based on their orientable genus:

\[
\partial_{\Gamma_{H}}(z)=\sum\limits_{A\subseteq E(H)}z^{\gamma(H^{A})}.
\]\end{defn}

\begin{defn}\label{def6}
A polynomial \( f(z) = \sum\limits_{i=0}^{m} a_{i} z^{i} \) is said to be \textit{interpolating}\cite{GMT20} if its spectrum, denoted as \( Spec(f) \), forms an integer interval \( [m,n] \) that includes all integers from \( m \) to \( n \). Here, the spectrum of the polynomial \( f(z) \), defined as \( Spec(f) = \{ i\;|\;a_{i} \neq 0\} \), represents the indices of non-zero coefficients.

A gap in the spectrum of the polynomial \( f(z) \) is characterized as a maximal integer interval \( [m,n] \), such that the intersection of this interval with the spectrum satisfies \( [m,n] \cap Spec(f)=\emptyset,\; n\geq m+2. \) The size of such a gap is given by \( n-m-1\), which denotes the number of integers contained within this gap.\end{defn}

Consequently, the spectrum of the partial-dual Euler-genus polynomial of a hypermap is given by \( Spec(\partial_{\varepsilon_{H}}) = \{\varepsilon(H^{A}), A \subset H\} \); similarly, the spectrum of the partial-dual orientable genus polynomial of a hypermap is expressed as \( Spec(\partial_{\Gamma_{H}}) = \{\gamma(H^{A}), A \subset H\} \).
\begin{prop}\label{pro2}
 Let $H$ be a connected hypermap. Then,

 (1) $\partial_{\Gamma_{H}}(1)=2^{e(H)};$

 (2) If $H$ is orientable, then $\partial_{\Gamma_{H}}(z)=\partial_{\varepsilon_{H}}(z^{2});$

 (3) The size of a gap in $Spec(\partial_{\varepsilon_{H}})$ can tend to arbitrarily large;

 (4) $Spec(\partial_{\Gamma_{H}})$ is not necessarily an interpolating polynomial.
\end{prop}
\begin{proof}
(1) We have $\partial_{\Gamma_{H}}(1)=\sum\limits_{A\subseteq E(H)}1^{\gamma(H^{A})}=\sum\limits_{A\subseteq E(H)}1=2^{e(H)}.$

(2) If $H$ is oriented, then it follows that $\varepsilon(H)=2\gamma(H)$, leading to the expression $\partial_{\varepsilon_{H}}(z)=\sum\limits_{A \subseteq E(H)} z^{\varepsilon(H^{A})} = \sum\limits_{A \subseteq E(H)} z^{2 \gamma(H^{A})}.$

(3) As illustrated in Figure \ref{p9}, we find that
$Spec(\partial_{\varepsilon_{H}}) = \{\varepsilon(H^{A})\} = \{0,4,2n-8,2n-4\}.$
Thus, the size of this gap is given by {$2n-8-4-1=2n-13,\; n \geq 7$}, which becomes arbitrarily large as $n$ approaches arbitrarily large.


(4) As depicted in Figure \ref{p6}, according to Theorem \ref{th3}, we obtain $Spec(\partial_{\Gamma_{H}})=\{0, 2, 3\}$. It is evident that $\partial_{\Gamma_H}(z)$ does not constitute an interpolating polynomial.

\end{proof}
\begin{figure}[H]
    \centering
    \includegraphics[width=2.5in]{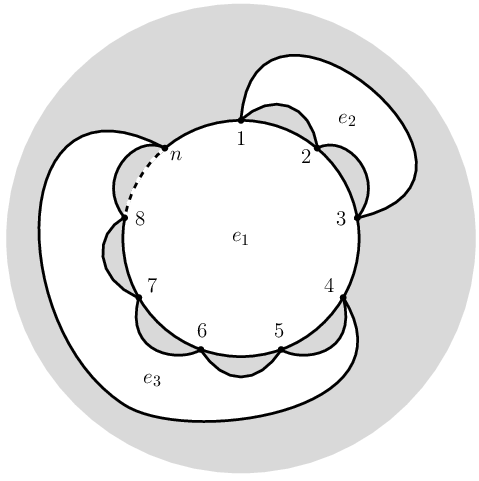}
    \caption{The size of the gap can be arbitrarily large.}
    \label{p9}

\end{figure}

\begin{rem}Gross, Mansour, and Tucker {\cite{GMT20}} demonstrated that the partial-dual genus polynomials for all orientable maps are interpolating. However, as indicated in (3) of Proposition \ref{pro2}, this assertion does not hold true for hypermaps.
\end{rem}

\section{The operations of join and bar-amalgamation in hypermaps}
In this section, we present the operations of join and bar-amalgamation for hypermaps, along with the corresponding partial-dual Euler-genus polynomials associated with these two operations.

\begin{defn}\label{def7}
Given two connected and disjoint hypermaps \(H_{1} = (B_{1}, \tau_{1}, \psi_{1})\) and \( H_{2} = (B_{2}, \tau_{2}, \psi_{2}) \), for any vertex \( v_{1} \in H_{1} \) and \( v_{2} \in H_{2} \), the corresponding bi-rotations are defined as follows:
\[
\tau_{1}(v_{1}) = (i_1, i_2, \ldots, i_k)(i'_{1},i'_{k}, i'_{k-1},\ldots,i'_{2});
\quad
\tau_2(v_2) = (j_1, j_2, \ldots, j_m)(j'_{\text{1}},j'_{m}, j'_{m-1},\ldots,j'_{2})
\]
where \( i_k, i'_k, j_\text{m}, j'_\text{m} \) are positive integers. We define the \textit{join operation} on \( H_1 \) and \( H_2 \) with respect to vertices \( v_1 \) and \( v_2\), denoted by \(H_1\vee H_2\), as follows:

(1) The functions $\tau_{1}(v)$ and $\tau_{2}(v')$ remain invariant for any vertices $v_{1} \neq v \in H_{1}$ and $v_{2} \neq v' \in H_{2}$.

(2) For any corner $i_{x}i^{'}_{x}$ of vertex $v_{1}$, the gluing of vertex $v_{2}$ results in a new vertex denoted as $u$. The bi-rotation of this new vertex is expressed as follows:
$$\tau(u) = (i_{1}, i_{2}, \ldots, i_{x-1}, j_{1}, j_{2}, \ldots, j_{m}, i_x, \ldots, i_k)(i^{'}_1, i^{'}_k, \ldots, i^{'}_{x+1}, j^{'}_1, j^{'}_m,\ldots,j^{' }_2,i^{' }_x,\ldots,i^{' }_2).$$
\end{defn}

Figure \ref{p10} illustrates the join \( H_{1} \vee H_{2} \) of the two hypermaps \( H_1 \) and \( H_2 \).
\begin{figure}[H]
    \centering
    \includegraphics[width=4in]{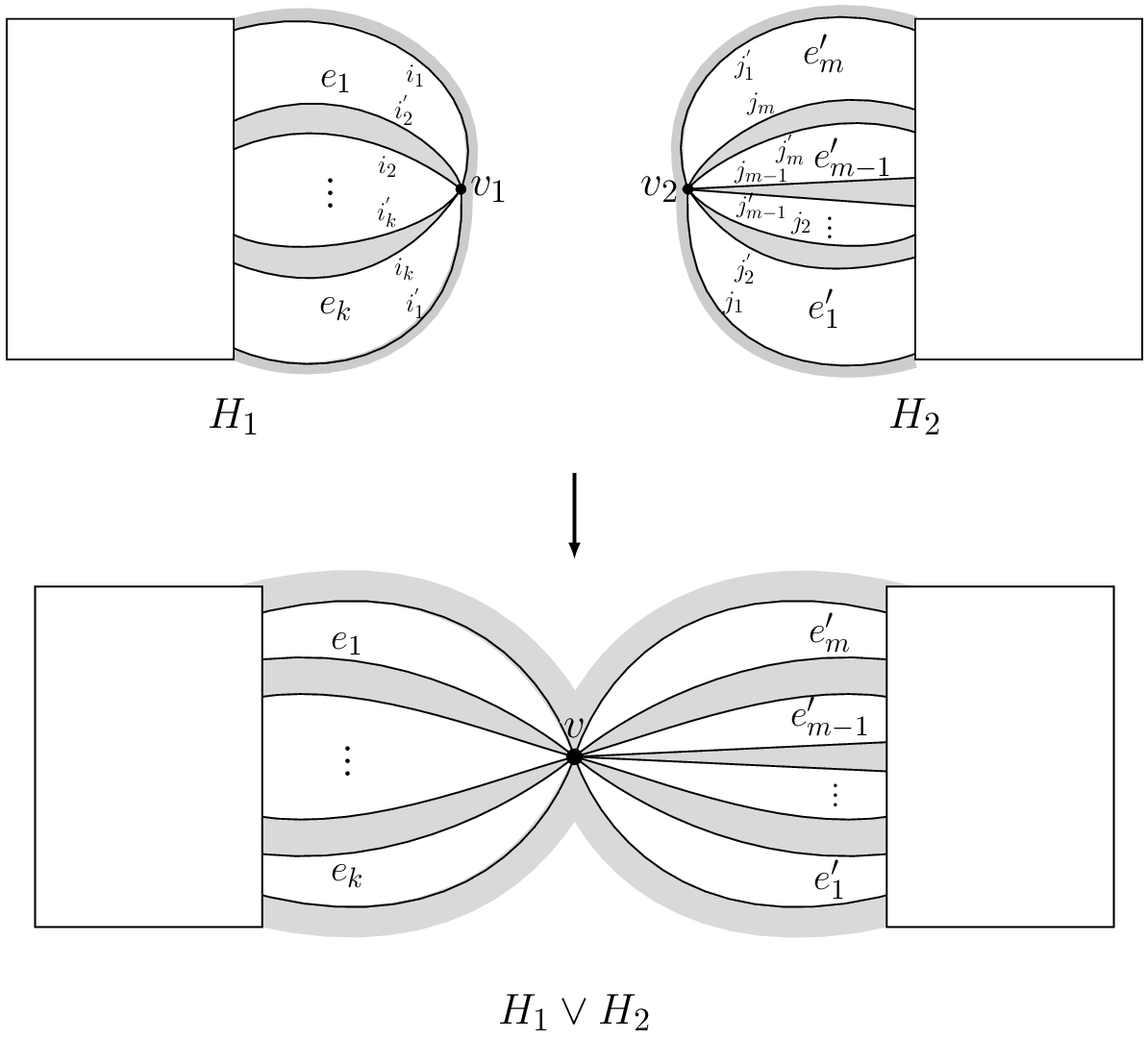}
    \caption{The join \( H_{1} \vee H_{2} \) of the hypermaps \( H_1 \) and \( H_2 \).}
    \label{p10}
\end{figure}
\begin{prop}\label{pro3}
Let \( H_{1} \) and \( H_{2} \) be two connected hypermaps such that \( H_{1} \cap H_{2} = \emptyset \). Then, we have the following results:

\begin{enumerate}
  \item The Euler characteristic of the disjoint union of these hypermaps is given by
\[
\chi(H_{1} \vee H_{2}) = \chi(H_{1}) + \chi(H_{2}) - 2.
\]

  \item  The Euler-genus in the disjoint union can be expressed as
\[
\varepsilon(H_{1} \vee H_{2}) = \varepsilon(H_{1}) + \varepsilon(H_{2}).
\]
\end{enumerate}

\end{prop}

\begin{proof}
According to Definition \ref{def7}, we can derive the following equations:

\begin{align*}f(H_{1}\vee H_{2})&=f(H_{1})+f(H_{2})-1,\\
v(H_{1}\vee H_{2})&=v(H_{1})+v(H_{2})-1,\\
e(H_{1}\vee H_{2})&=e(H_{1})+e(H_{2}),\\
\sum\limits_{i=1}^{e(H_{1}\vee H_{2})}n_{i}(H_{1}\vee H_{2})&=\sum\limits_{i=1}^{e(H_{1})}n_{i}(H_{1})+\sum\limits_{i=1}^{e(H_{2})}n_{i}(H_{2}).
\end{align*}
From these, the Euler characteristic of the hypermap is expressed as follows:
    \begin{equation}
    \begin{aligned}
    \chi(H_{1}\vee H_{2})&=v(H_{1}\vee H_{2})+f(H_{1}\vee H_{2})+e(H_{1}\vee H_{2})-\sum\limits_{i=1}^{e(H_{1}\vee H_{2})}n_{i}(H_{1}\vee H_{2})\\
    &=\left[v(H_{1})+f(H_{1})+e(H_{1})-\sum\limits_{i=1}^{e(H_{1})}n_{i}(H_{1})\right]\\& \ \ +\left[v(H_{2})+f(H_{2})+e(H_{2})-\sum\limits_{i=1}^{e(H_{2})}n_{i}(H_{2})\right]-2\\
    &=\chi(H_{1})+\chi(H_{2})-2\nonumber.
    \end{aligned}
    \end{equation}
Consequently, we have
 $\varepsilon(H_{1}\vee H_{2})=2c(H_{1}\vee H_{2})-\chi(H_{1}\vee H_{2})=2-\chi(H_{1}\vee H_{2})=2-\chi(H_{1})+2-\chi(H_{2})=\varepsilon(H_{1})+\varepsilon(H_{2}).$
\end{proof}

\begin{thm}\label{th4}
Let \( H_{1} \) and \( H_{2} \) be two connected hypermaps such that \( H_{1} \cap H_{2} = \emptyset \). Then, the partial-dual Euler-genus polynomial of the join \( H_{1} \vee H_{2} \) is given by
\[
\partial_{\varepsilon_{H_{1}\vee H_{2}}}(z) = \partial_{\varepsilon_{H_{1}}}(z) \cdot \partial_{\varepsilon_{H_{2}}}(z).
\]\end{thm}
\begin{proof}
Note that \( E(H_{1} \vee H_{2}) = E(H_{1}) \cup E(H_{2}) \). Let \( A \subseteq E(H_{1} \vee H_{2}) \). We can express the set \( A \) as follows:
\[
A = [A \cap E(H_{1})] \cup [A \cap E(H_{2})] \triangleq A_{1} \cup A_{2},
\]
where it is clear that \( A_{1} \cap A_{2} = \emptyset. \)

Consequently, we have the relationship:
\[
\varepsilon[(H_{1} \vee H_{2})^{A}]=\varepsilon[H_{1}^{A_ {1}}]+\varepsilon[H_ {2}^{A_ {2}}].
\]

Thus, we arrive at the conclusion:\begin{equation}
        \begin{aligned}
            \partial_{\varepsilon_{H_{1}\bigvee H_{2}}}(z)&=\sum\limits_{A\subseteq E(H_{1}\vee H_{2})}z^{\varepsilon[(H_{1}\vee H_{2})^{A}]}\\
                                                      &=\sum\limits_{A\subseteq E(H_{1}\vee H_{2})}z^{\varepsilon[H_{1}^{A\cap E(H_{1})}]+\varepsilon[H_{2}^{A\cap E(H_{2})}]}\\
                                                      &=\sum\limits_{A\subseteq E(H_{1}\vee H_{2})}z^{\varepsilon[H_{1}^{A\cap E(H_{1})}]}
                                                      \cdot z^{\varepsilon[H_{2}^{A\cap E(H_{2})}]}\\
                                                      &=\sum\limits_{A_{1}\subseteq E(H_{1})}z^{\varepsilon[H_{1}^{A_{1}}]}\sum\limits_{A_{2}\subseteq E(H_{2})}z^{\varepsilon[H_{2}^{A_{2}}]}\\
                                                      &=\partial_{\varepsilon_{H_{1}}}(z)\partial_{\varepsilon_{H_{2}}}(z)\nonumber
        \end{aligned}
    \end{equation}
\end{proof}

\begin{defn}
\label{def8}
Let \( H_{1} = (B_{1}, \tau_{1}, \psi_{1}) \) and \( H_{2} = (B_{2}, \tau_{2}, \psi_{2}) \) be two connected hypermaps, with the condition that \( H_{1} \cap H_{2} = \emptyset \). For any hyperedge \( e_{1} \in H_{1} \) and \( e_{2} \in H_{2} \), let us select \( m \) vertices from \( e_1\), denoted as \( v_1, v_2, ..., v_m\); similarly, let us choose \( n\) vertices from  \( e_2\), represented by  \( u_1, u_2, ..., u_n\). The bi-rotation of each vertex is defined as follows:
\[\tau_{1}(v_{i})=(i_{1},i_{2},\ldots,i_{k})(i^{'}_{1},i^{'}_{k},\ldots,i^{'}_{2}),\; i=1,2,\ldots,m;\]
   \[\tau_{2}(u_{j})=(j_{1},j_{2},\ldots,j_{t})(j^{'}_{1},j^{'}_{t},\ldots,j^{'}_{2}),\; j=1,2,\ldots,n.\]
We define the \textit{bar-amalgamation} operation on $H_{1}$ and $H_{2}$, denoted by $H_{1}\oplus_e H_{2}$, as follows:

(1) For any corner \(i_{x_i} i'_{x_{i'}}\) of vertex \(v_i\), where \(i=1,2,\ldots,m\) and \(x_i \in \{ 1,2,\ldots,k\}\);

(2) Select any corner \(j_{y_j} j'_{y_{j'}}\) of vertex \(u_j\), where \(j=1,2,\ldots,n\) and \(y_j \in \{ 1,2,\ldots,t\}\);

(3) Construct a new hyperedge \(e\), which connects all the corners identified in (1) and (2). This hyperedge is labeled by indices \(i''_{x_i}, i'''_{x_{i'}}, j''_{y_j}, j'''_{y_{j'}}\), such that
\[
\psi(e)= (1^{''}_{x_1}, 2^{''}_{x_2}, \ldots, m^{''}_{x_m}, n^{''}_{y_n}, \ldots, 2^{''}_{y_2}, 1^{''}_{y_1})(1^{''' } _{ x _ {  i'}} ,  1 ^ { ''' } _ { y _ { j'}} ,   2 ^ { ''' } _ { y _ {  j'}} ,   \ldots , n ^ {'''} _ { y_{n'}} , m ^ {'''} _ { x_{m'}} ,   \ldots ,    2 ^ {'''} _{ x_{i' }} )
\]
The newly constructed hyperedge \(e\) is referred to as the connecting hyperedge.
\end{defn}

Clearly, we  have:
\[B(H_{1}\oplus_e H_{2})=B_{1}\cup B_{2}\cup\{i^{''}_{x_{i}}, i^{'''}_{x_{i^{'}}}, j^{''}_{y_{j}}, j^{'''}_{y_{j^{'}}}\}, \; \psi(H_{1}\oplus_e H_{2})=\psi_{1}\cup\psi_{2}\cup\psi(e);\]
\[\tau(H_{1}\oplus_e H_{2})=\{(i_{1},i_{2},\ldots,i_{x_{i}-1},i^{'''}_{x_{i^{'}}},i_{x_{i}},\ldots,i_{k})(i^{'}_{1},i^{'}_{k},\ldots,i^{'}_{x_{i^{'}}+1},i^{''}_{x_{i}},i^{'}_{x_{i^{'}}},\ldots,
i^{'}_{2})\}\]
\[\cup\{(j_{1},j_{2},\ldots,j_{y_{j}-1},j^{'''}_{y_{j^{'}}},j_{y_{j}},\ldots,j_{t})
(j^{'}_{1},j^{'}_{t},\ldots,j^{'}_{y_{j^{'}}+1},j^{''}_{y_{j}},j^{'}_{y_{j^{'}}},\ldots,j^{'}_{2})\}\]
Where $i=1,2,\ldots,m$ and $j=1,2,\ldots,n.$

\begin{example}
\begin{figure}[H]
    \centering
    \includegraphics[width=4in]{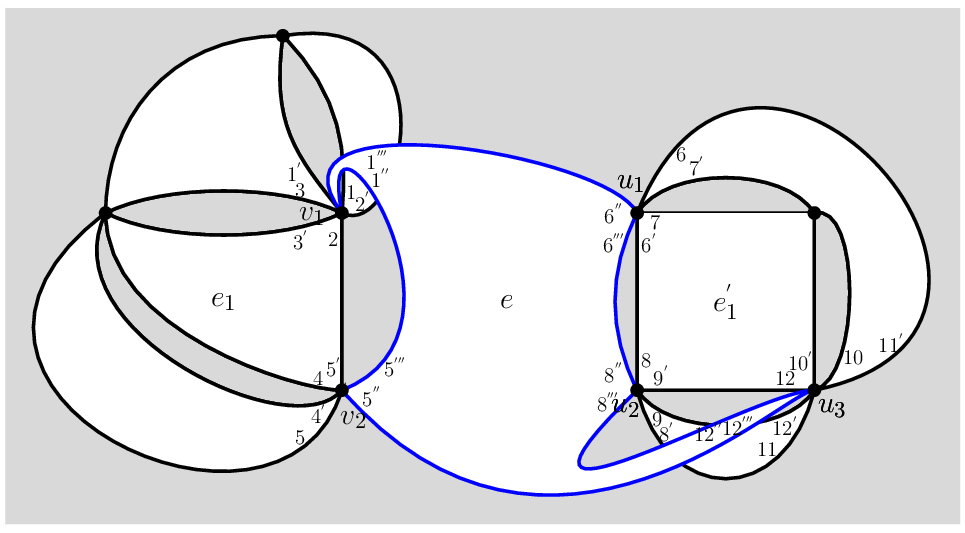}
    \caption{The bar-amalgamation $H_{1}\oplus_e H_{2}$.}
    \label{p11}
\end{figure}

By the definition of bar-amalgamation, we can have the results as follows:
\begin{align*}
\psi(e)&=(1^{'''},6^{'''},8^{'''},12^{'''},5^{'''})(1^{''},5^{''},12^{''},8^{''},6^{''})\\
\tau(v_{1})&=(1,2,3,1^{'''})(1^{'},3^{'},2^{'},1^{''})\\
\tau(v_{2})&=(4,5^{'''},5)(4^{'},5^{''},5^{'})\\
\tau(u_{1})&=(6,7,6^{'''})(6^{'},7^{'},6^{''})\\
\tau(u_{2})&=(8,9,8^{'''})(8^{'},9^{'},8^{''})\\
\tau(u_{3})&=(10,11,12^{'''},12)(10^{'},12^{''},12^{'},11^{'}).
\end{align*}
The bi-rotations of the other vertices remain consistent.
\end{example}
\begin{prop}\label{pro4}Let \( H_{1} \) and \( H_{2} \) be two connected hypermaps, with the condition that \( H_{1} \cap H_{2} = \emptyset \). Then we have the following results:
\begin{enumerate}
  \item If there are \( m \) corners in \( k_{1} \) distinct faces of \( H_{1} \), and \( n \) corners in \( k_{2} \) distinct faces of \( H_{2} \), then the function for the combined hypermap is given by \[
f(H_{1}\oplus_e H_{2}) = f(H_{1}) + f(H_{2}) + m + n - 2(k_{1}+k_{2}) + 1,
\]  where $1\leq k_{1}\leq m, 1\leq k_{2}\leq n$.

  \item The Euler characteristic of the combined hypermap can be expressed as $\chi(H_{1}\oplus_e H_{2})=\chi(H_{1})+\chi(H_{2})+2-2k_{1}-2k_{2}.$

  \item  $\varepsilon(H_{1}\oplus_e H_{2})=\varepsilon(H_{1})+\varepsilon(H_{2})+2k_{1}+2k_{2}-4.$

\end{enumerate}

\end{prop}

\begin{proof}
By the definition of $H_{1}\oplus_e H_{2}$, we have: $v(H_{1}\oplus_e H_{2})=v(H_{1})+v(H_{2})$, $e(H_{1}\oplus_e H_{2})=e(H_{1})+e(H_{2})+1$, $\sum\limits_{i=1}^{e(H_{1}\oplus_e H_{2})}n_{i}(H_{1}\oplus_e H_{2})=\sum\limits_{i=1}^{e(H_{1})}n_{i}(H_{1})+\sum\limits_{i=1}^{e(H_{2})}n_{i}(H_{2})+m+n$.

According to Definition~\ref{def3}, items (2) and (3) follow from item (1). We now provide a proof for item (1). There are three cases.


(1) If \( k_{1} = m \) and \( k_{2} = n \), the new hyper-edge \( e \) merges \( m \) faces from \( H_{1} \) and \( n \) faces from \( H_{2} \) into a single face. Consequently, we have:
\[
f(H_{1}\oplus_e H_{2}) = f(H_{1}) + f(H_{2}) - m - n + 1;
\]

(2)  If both \( k_{1} \) and \( k_{2} \) are equal to $1,$ then:
\[
f(H_{1}\oplus_e H_{2}) = f(H_{1}) + f(H_{2}) + m + n - 3;
\]

   (3) If \(1 < k_{1} < m\) and \(1 < k_{2} < n\). We define two distinct sets of faces: the first set consists of faces \(f_{1}, f_{2}, \ldots, f_{k_{1}}\), while the second set comprises faces \(F_{1}, F_{2}, \ldots, F_{k_{2}}\). Next, we select non-overlapping segments \(x_i\) from each face \(f_i\) for all \(i = 1, 2, \ldots, k_1\), and non-overlapping segments \(y_j\) from each face \(F_j\) for all \(j = 1, 2, \ldots, k_2\). It follows that the sums satisfy the conditions:
\[
\sum_{i=1}^{k_1} x_i = m,
\]
and
\[
\sum_{j=1}^{k_2} y_j = n.
\]
\begin{figure}[H]
    \centering
     \includegraphics[width=4in]{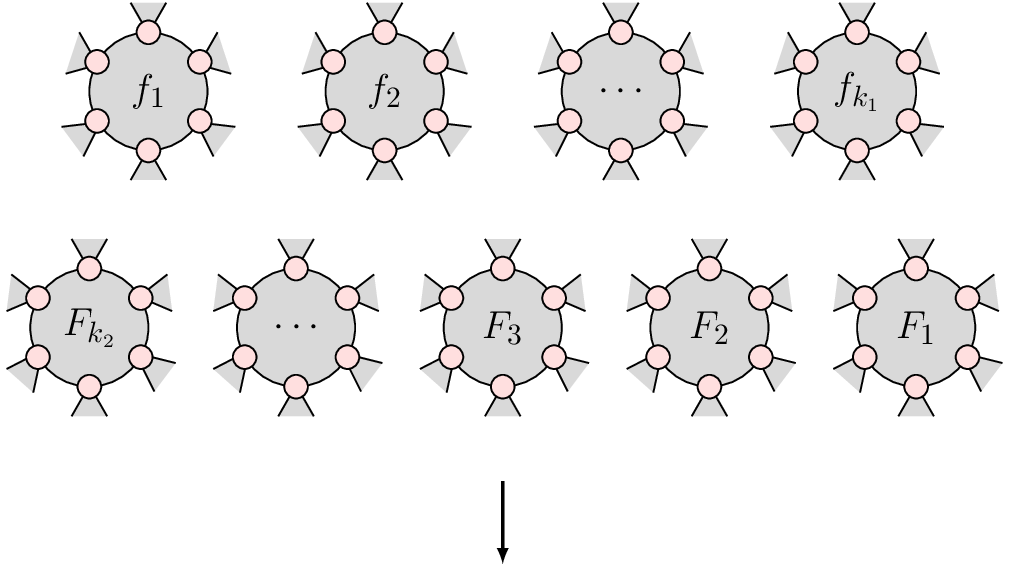}
    \includegraphics[width=4in]{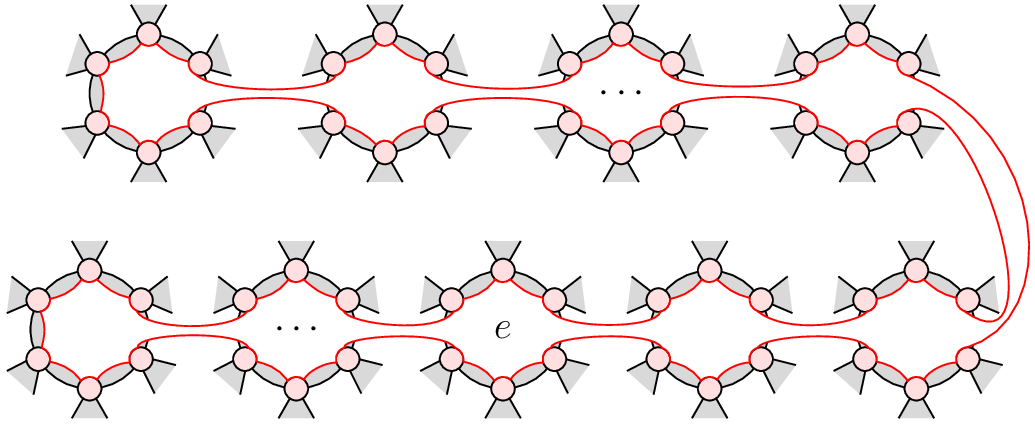}
    \caption{Faces of $H_{1}\oplus_e H_{2}$.}
    \label{p12}

\end{figure}

Therefore, the new edge \( e \) forms \( x_{i}-1 \) distinct faces with the boundary of face \( f_{i} \) in \( H_{1} \), resulting in a total of
\[
\sum\limits_{i=1}^{k_{1}}(x_{i}-1)=m-k_{1}
\]
faces. Similarly, the new edge \( e \) creates \( y_{j}-1 \) distinct faces with the boundary of face \( F_{j} \) in \( H_{2} \), leading to a total of
\[
\sum\limits_{j=1}^{k_{2}}(y_{j}-1)=n-k_{2}
\]
faces, where there is an  a common face among the combined count of \( k_{1}+k_{2} \) faces, as shown in Figure \ref{p12}. Consequently, we can express the number of faces for the connected sum as follows:
\[
f(H_ { 1 } \oplus_e H_ { 2 }) = f(H_ { 1 }) + f(H_ { 2 }) - (k_ { 1 } + k_ { 2 }) + [m+n-(k_ { 1 }+k_ { 2 }-1)].
\]
This simplifies to:

\[
f(H _{ 1 }\oplus_e H _{ 2 }) = f(H _{ 1 }) + f(H _{ 2 }) + m+n - 2(k _{ 1 }+k _{ 2}) + 1.
\]
\end{proof}

\begin{thm}\label{th6}
Let \( H_{1} \) and \( H_{2} \) be two connected hypermaps such that \( H_{1} \cap H_{2} = \emptyset \). Then, we have the following equation:

\begin{equation}
    \partial_{\varepsilon_{H_{1}\oplus_e H_{2}}}= 2\sum\limits_{A\subseteq E(H_{1}\cup H_{2})} z^{\varepsilon[H_{1}^{A\cap E(H_{1})}] + \varepsilon[H_{2}^{A\cap E(H_{2})}] + 2(k_1 + k_2 - 2)}.
\nonumber
\end{equation}

Here, \( k_1 \) denotes the number of distinct faces associated with the \( m \) corners located in \( A^{c}\cap H_1 \), while \( k_2 \) represents the number of distinct faces corresponding to the \( n \) corners situated in \( A^{c}\cap H_2 \).\end{thm}

\begin{proof}
Let \( A_{1} = A \cap E(H_{1}) \) and \( A_{2} = A \cap E(H_{2}) \). The discussion regarding whether the edge \( e \) belongs to set \( A \) is presented below. There are two cases to consider.

\begin{enumerate}
  \item If \( e \notin A \), then we have \( A = A_{1} \cup A_{2} \), \( A^{c} = A_{1}^{c} \cup A_{2}^{c} \cup \{e\} \), and \( A_{1} \cap A_{2} = \emptyset. \) By Property \ref{property1}, it follows that:

\begin{equation}
    \begin{aligned}
        v[(H_{1}\oplus_e H_{2})^{A}]&=f(A)=f(A_{1})+f(A_{2})=v(H_{1}^{A_{1}})+v(H_{2}^{A_{2}}),\\
        e[(H_{1}\oplus_e H_{2})^{A}]&=e(H_{1}\oplus_e H_{2})=e(H_{1})+e(H_{2})+1=e(H_{1}^{A_{1}})+e(H_{2}^{A_{2}})+1,\\
        f[(H_{1}\oplus_e H_{2})^{A}]&=f(A^{c}),\\
        \sum\limits_{i=1}^{e[(H_{1}\oplus_e H_{2})^{A}]} n_{i}[(H_{1}\oplus_e H_{2})^{A}]&=\sum\limits_{i=1}^{e(H_{1}\oplus_e H_{2})} n_{i}(H_{1}\oplus_e H_{2})=\sum\limits_{i=1}^{e(H_{1})} n_{i}(H_{1})+\sum\limits_{i=1}^{e(H_{2})} n_{i}(H_{2})+m+n.\nonumber
    \end{aligned}
\end{equation}

If \( A^{c} = A_{1}^{c} \cup A_{2}^{c} \cup \{e\} \) is not connected, we can analyze the connected components. According to Proposition \ref{pro4}, we have:$
f(A^{c}) = f(A_{1}^{c}) + f(A_{2}^{c}) + m + n - 2k_{1} - 2k_{2} + 1 = f(H_{1}^{A_{1}}) + f(H_{2}^{A_{2}}) + m + n - 2k_{1} - 2k_{2} + 1
$, where \( k_1 \) denotes the number of distinct faces associated with the \( m \) corners in \( A^{c}\cap H_1 \), and \( k_2 \) represents the number of distinct faces linked to the \( n \) corners in \( A^{c}\cap H_2 \).

In conclusion,
$\chi[(H_{1}\oplus_e H_{2})^{A}]=v[(H_{1}\oplus_e H_{2})^{A}]+e[(H_{1}\oplus_e H_{2})^{A}]+f[(H_{1}\oplus_e H_{2})^{A}]-\sum\limits_{i=1}^{e[(H_{1}\oplus_e H_{2})^{A}]} n_{i}[(H_{1}\oplus_e H_{2})^{A}]=v(H_{1}^{A_{1}})+v(H_{2}^{A_{2}})+e(H_{1}^{A_{1}})+e(H_{2}^{A_{2}})+1+f(H_{1}^{A_{1}})+f(H_{2}^{A_{2}})+m+n-2k_{1}-2k_{2}+1-\sum                                     \limits_{i=1}^{e[(H_{1}\oplus_e H_{2})^{A}]}n_{i}[(H_{1}\oplus_e H_{2})^{A}]=\chi(H_{1}^{A_{1}})+\chi(H_{2}^{A_{2}})-2k_{1}-2k_{2}+2.$

\vspace{0.3cm}
Thus, we have
\begin{equation}
    \begin{aligned}
        \varepsilon[(H_{1}\oplus_e H_{2})^{A}]&=2c[(H_{1}\oplus_e H_{2})^{A}]-\chi[(H_{1}\oplus_e H_{2})^{A}]\\
                                            &=2c(H_{1}\oplus_e H_{2})-\chi(H_{1}^{A_{1}})-\chi(H_{2}^{A_{2}})+2k_{1}+2k_{2}-2\\
                                            &=2c(H_{1})+2c(H_{2})-2-\chi(H_{1}^{A_{1}})-\chi(H_{2}^{A_{2}})+2k_{1}+2k_{2}-2\\
                                            &=\varepsilon[(H_{1})^{A_{1}}]+\varepsilon[(H_{1})^{A_{1}}]+2k_{1}+2k_{2}-4.\nonumber
    \end{aligned}
\end{equation}

  \item If \( e \in A \), then we have \( A = A_{1} \cup A_{2} \cup \{e\} \) and \( A^{c} = A^{c}_{1} \cup A^{c}_{2} \). In this case, \( A \) corresponds to \( A^{c} \) as described in situation (1). Therefore, we can express the relationship as follows:
\[
\varepsilon[(H_{1}\oplus_e H_{2})^{A}] = \varepsilon[(H_{1})^{A_{1}}] + 2k_{1} + 2k_{2} - 4,
\]
where \( k_{1} \) denotes the number of distinct faces that contain the \( m \) corners located in \( A_{1} \), and \( k_{2} \) represents the number of distinct faces that include the \( n \) corners situated in \( A_{2}.\)
Thus, we have
 \begin{equation}
    \begin{aligned}
        \partial_{\varepsilon_{H_{1}\oplus_e H_{2}}}&=\sum\limits_{A\subseteq E(H_{1}\oplus_e H_{2})}z^{\varepsilon[(H_{1}\oplus_e H_{2})^{A}]}\\
                                                  &=2\sum\limits_{A\subseteq E(H_{1}\cup H_{2})}z^{\varepsilon[H_{1}^{A\cap E(H_{1})}]+\varepsilon[H_{2}^{A\cap E(H_{2})}]+2(k_{1}+k_{2}-2)}\nonumber
    \end{aligned}
 \end{equation}

\end{enumerate}

\end{proof}

\section{Subdivision of a hyperedge in a hypermap}
In this section, we present the operation of subdivision for hypermaps, along with the corresponding partial-dual Euler-genus polynomials that are associated with this operation for a hyperedge \( e \) where \( n(e) = 3 \).

\begin{defn}
\label{def10}
Given a connected hypermap \( H \), for any hyperedge \( e \) in \( H \), if the number of vertices incident to \( e \) is denoted as \( n(e) = k \), then the subdivision of the hyperedge \( e \), represented by \( \bar{H} \), can be described as follows:
\begin{enumerate}
  \item Introduce a new vertex, denoted as \( u \);
  \item Select any \( k \) vertices from the hyperedge \( e \) and construct a new hyperedge that includes the vertex \( u \);
  \item Remove the original hyperedge \( e \).
\end{enumerate}
\end{defn}
By definition, we have $v(\bar{H})=v(H)+1,\; e(\bar{H})=e(H)+\binom{k}{k-1}-1=e(H)+k-1$.

Figure \ref{p16} illustrates a subdivision of a hyperedge $e$ within a hypermap $H$, where the number of elements in the hyperedge $3$.

\begin{figure}[H]
    \centering
     \includegraphics[width=4.5in]{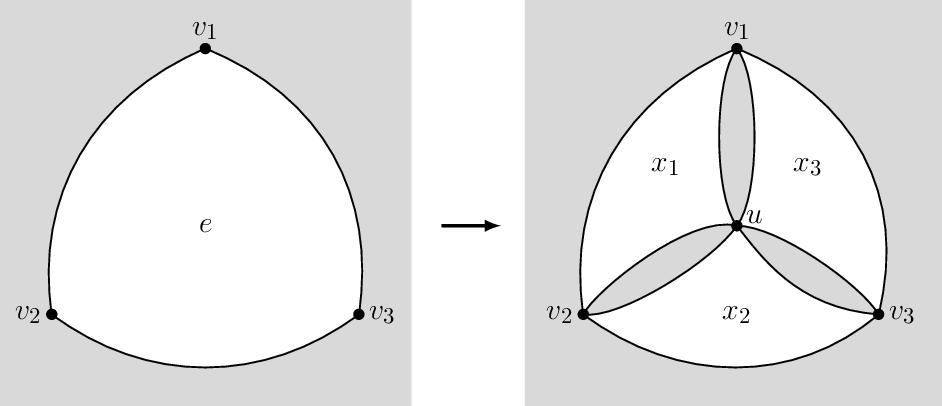}
\caption{A subdivision of a hyperedge $e$ with $n(e)=3$.}
    \label{p16}

\end{figure}

\begin{lemma}\label{th6}
Given a connected hypermap \( H \), let \( e \in E(H) \) with \( n(e) = 3 \). Then, we have
\(
\varepsilon(\bar{H}) = \varepsilon(H).
\)\end{lemma}

\begin{proof}
Let \( G(\bar{H}) \) be the bipartite graph corresponding to \( \bar{H} \). According to Definition \ref{def10}, we have:
\[
v(\bar{H}) = v(H) + 1, \quad e(\bar{H}) = e(H) + 2, \quad f' = f(\bar{H}) - f(H),
\]
and it follows that \( f' \leq 3. \)

By Definition \ref{def3}, we have
\begin{equation}
    \begin{aligned}
        \chi(\bar{H})&=v(\bar{H})+f(\bar{H})+e(\bar{H})-\sum\limits_{i=1}^{e(\bar{H})}n_{i}(\bar{H})\\
                           &=v(H)+1+f(H)+f'+e(H)+2-\sum\limits_{i=1}^{e(H)}n_{i}(H)-6\\
                           &=\chi(H)+f'-3\nonumber.
    \end{aligned}
\end{equation}

Since the subdivision for \( e \) is illustrated in Figure \ref{p16}, it follows that \( f' = 3 \). Therefore, we have \( \chi(\bar{H}) = \chi(H) \), which implies that \( \varepsilon(\bar{H}) = \varepsilon(H) \).\end{proof}

\begin{thm}\label{co1}
Given a connected hypermap \( H \) with an edge \( e \in E(H) \) such that \( n(e) = 3 \), we have the following expression for the subdivision operator:

\[
\partial_{\varepsilon_{\bar{H}}}(z) = \sum_{A \subseteq E(H)} (a_i z^{\varepsilon[H^A]} + b_i z^{\varepsilon[H^A] + 2} + c_i z^{\varepsilon[H^A] + 4}),
\]

where it holds that

\[
\sum_{i=1}^{2^{e(H)}} (a_i + b_i + c_i) = 2^{e(H)+2},
\]
with the conditions \( a_i, b_i, c_i \geq 0. \)
\end{thm}

\begin{proof}
 Let the hyperedge \( e \) be \(\{v_{1}, v_{2}, v_{3} \} \). Upon subdividing the hyperedge \( e \), we introduce new hyperedges and a new vertex, denoted as \( x_{1}, x_{2}, x_{3} \), and \( u \), respectively. Consequently, it follows that \( v(\bar{H}) = v(H) + 1 \) and \( e(\bar{H}) = e(H) + 2 \).
Consider any subset of hyperedges \( \bar{A} \subseteq E(\bar{H}) \). We can then define the corresponding set \( A = \{e_{i} \in E(H) | e_{i} \in \bar{A}\} \). It follows that \( A\subseteq\bar{A} \) and that the hyperedge \( e\) is not included in set \( A\).

By Theorem \ref{th3} and Lemma \ref{th6}, we have

\[
\varepsilon(\bar{H}^{\bar{A}}) - \varepsilon(H^{A}) = 2[c(A) + c(A^{c})] - 2[c(\bar{A}) + c(\bar{A}^{c})] + 2.
\]

Let us denote \( k = c(A) + c(A^{c}) \) and \( \bar{k} = c(\bar{A}) + c(\bar{A}^{c}) \). Thus, the equation can be reformulated as

\[
\varepsilon(\bar{H}^{\bar{A}}) - \varepsilon(H^{A}) = 2(k - \bar{k}) + 2.
\]


According to the symmetry of the partial duality of a hypermap, where $\varepsilon(H^{A})=\varepsilon(H^{A^c})$, we will examine two cases as follows.
 \begin{enumerate}
   \item If $\{x_{1}, x_{2}, x_{3}\} \cap \bar{A} = \emptyset$. Since \( e \notin A \), the vertex \( v_{i} \) (where \( 1 \leq i \leq 3 \)) belongs to the same connected component in both sets \( A \) and \( \bar{A} \). Furthermore, the new vertex $u$ is an isolated vertex in $\bar{A}$. Consequently, we have that \( c(\bar{A}) = c(A) + 1 \) and \( c(\bar{A}^{c}) = c(A^{c}) \). In other words, it follows that \( \bar{k} - k = 1. \)

   \item If $\bar{A}$ contains only one element from the set $\{x_{1}, x_{2}, x_{3}\}$, we categorize the number of connected components that include $v_{1}$, $v_{2}$, and $v_{3}$ in $A$ into the following three cases.
 \begin{description}
   \item[Case 1] When \( v_{1}, v_{2} \), and \( v_{3} \) belong to three distinct components within set \( A \), as illustrated in Figure \ref{p17}, we find that \( k - \bar{k} = 3 - 2 = 1 \). \begin{figure}[H]
    \centering
    \includegraphics[width=4in]{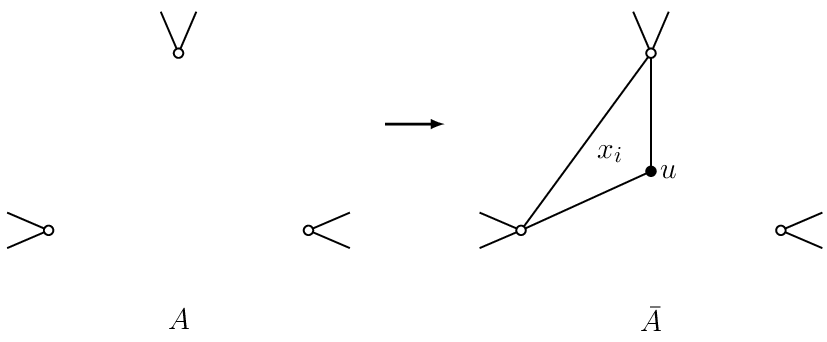}
    \caption{The three vertices \( v_{1} \), \( v_{2} \), and \( v_{3} \) are associated with three distinct components }
    \label{p17}
\end{figure}
  \item [Case 2]One vertex among \( v_1, v_2, v_3 \) lies in one component, while the other two vertices lie in a different component, as illustrated in Figure \ref{p18}. In this case, we can deduce that \( k - \bar{k} = 2 - 1 = 1 \), or equivalently \( k_{1} - k_{2} = 2 - 2 = 0 \).
  \begin{figure}[H]
    \centering
    \includegraphics[width=6in]{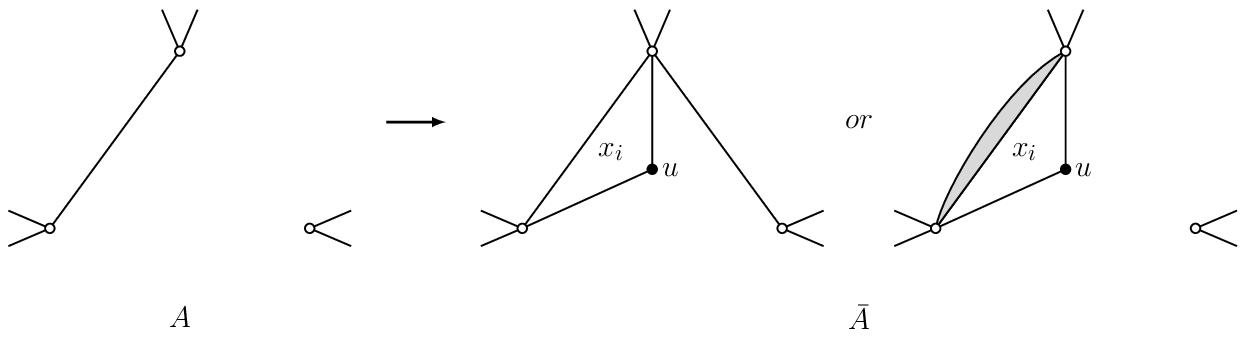}
    \caption{The three vertices $v_{1},v_{2}$ and $v_{3}$ belong to two different components.}
    \label{p18}
\end{figure}
\item [Case 3] When there is only one component that includes \(v_{1}, v_{2}\), and \(v_{3}\) in set \(A\), as illustrated in Figure \ref{p19}, we observe that \(k - \bar{k} = 1 - 1 = 0\).
 \begin{figure}[H]
    \centering
    \includegraphics[width=6in]{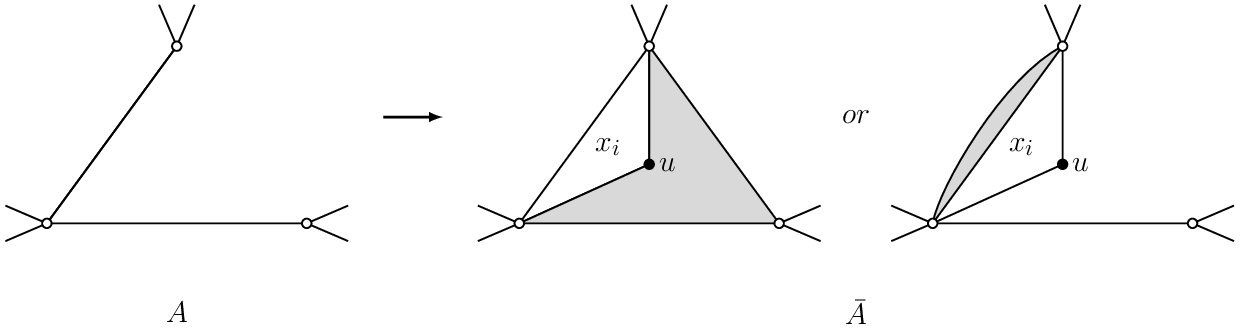}
    \includegraphics[width=4in]{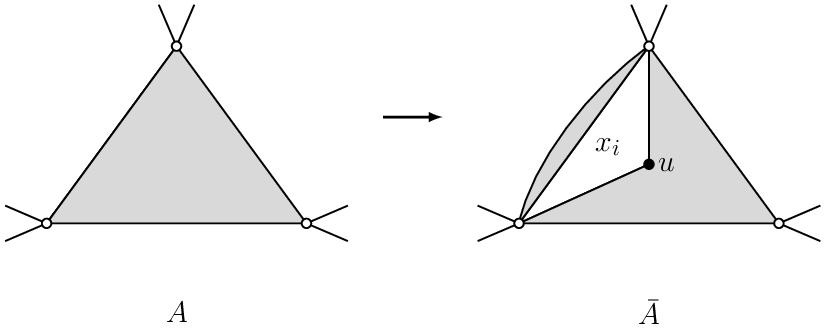}
    \caption{Only one component including vertices $v_{1},v_{2}$ and $v_{3}$.}
    \label{p19}
\end{figure}

 \end{description}
  \end{enumerate}

In conclusion, the expression $\varepsilon[(\bar{H})^{\bar{A}}]-\varepsilon(H^{A})$ can yield values of $0$, $2$, or $4$, which is represented as follows:

$$2(k-\bar{k}) + 2 = 0,\, 2 \text{ or } 4.$$

Since \( e \notin A \), there are \( 2^{e(H)-1} \) ways to select the set \( A \). Additionally, there are \( 2^{e(\bar{H})} = 2^{e(H)+2} \) ways to choose the set \( \bar{A} \). Therefore, we can express this as:

\[
\partial_{\varepsilon_{\bar{H}}}(z) = \sum\limits_{\bar{A}\subseteq E(\bar{H})} z^{\varepsilon[\bar{H}^{\bar{A}}]} = \sum\limits_{ e\notin A\subseteq E(H)} (a_{i} z^{\varepsilon[H^{A}]} + b_{i} z^{\varepsilon[H^{A}] + 2} + c_{i} z^{\varepsilon[H^{A}] + 4})
\]

where it holds that

\[
\sum_{i=1}^{2^{e(H)-1}} (a_{i}+b_{i}+c_{i}) = 2^{e(H)+2}, \quad a_{i}, b_{i}, c_{i}\geq 0.
\]\end{proof}

 \section{Applications}

\subsection{Hypertrees}
\begin{defn}\label{def9}Given a connected hypergraph $T,$ if the removal of any hyperedge from $T$ results in a disconnected hypergraph, then $T$ is defined as a hyper-tree.
\end{defn}
The classification of hypertrees can be categorized into two distinct types: hypertrees that contain cycles and those that are free of cycles.

\begin{thm}\label{th7}Let $T$ be a hypertree map without cycles. Then, the partial-dual Euler-genus polynomial of the hypertree map is given by

\[
\partial_{\varepsilon_{T}}(z) = 2^{e(T)}.
\]
\end{thm}
\begin{proof}In this context, the hypertree exhibits a structure analogous to that of trees within graphs. It can be derived by executing join operations on its connected components. According to Theorem \ref{th4}, this theorem is validated.
\end{proof}

The subsequent property is self-evident.
\begin{thm}\label{th8}
Let $H$ be a connected hyper-map, for each hyperedge $e\in E(H)$, adding a vertex to the hyperedge $e$ will not change the Euler-genus of the hypermap.
\end{thm}

\begin{example}
 Given a hypertree \( T = (V(T), E(T)) \), the vertex set \( V(T) \) consists of the vertices \( u, v_{1}, v_{2}, \ldots, v_{n-1}, x_{1}, x_{2}, \ldots, x_{n} \). The hyperedge set is defined as \( E(T) = \{ e_{1}, e_{2}, \ldots, e_{n} \} \), where each hyperedge is specified as follows: for \( i = 1, 2, \ldots, n-1\), we have \( e_{i} = \{ u, v_{i}, x_{i} \} \); and for the last hyperedge, we define \( e_{n} = \{ v_{1}, v_{2}, ..., v_{n-1}, x_n\}.\) Notably, all vertices \( x_1,x_2,\ldots,x_n\) are of degree one.

Figure \ref{p15} illustrates  the hypertree map \( T\). Consequently, the partial-dual Euler-genus polynomial of $T$ can be expressed as follows:
\[\partial_{\varepsilon_{T}}(z)=\begin{cases}
2+2z^{2n-4}+2\sum\limits_{i=1}^{[\frac{n}{2}]}\binom{n-1}{i}z^{2i}+2\sum\limits_{i=2}^{[\frac{n}{2}]}\binom{n-1}{i-1}z^{2n-2i}&\text{$n$ is odd},\\
2+2z^{2n-4}+2\sum\limits_{i=1}^{\frac{n}{2}-1}\binom{n-1}{i}z^{2i}+2\sum\limits_{i=2}^{\frac{n}{2}-1}\binom{n-1}{i-1}z^{2n-2i}+[\binom{n-1}
{\frac{n}{2}}+\binom{n-1}{\frac{n}{2}-1}]z^{n}&\text{$n$ is even}.
\end{cases}\]
\end{example}
\begin{figure}[H]
    \centering
     \includegraphics[width=2.6in]{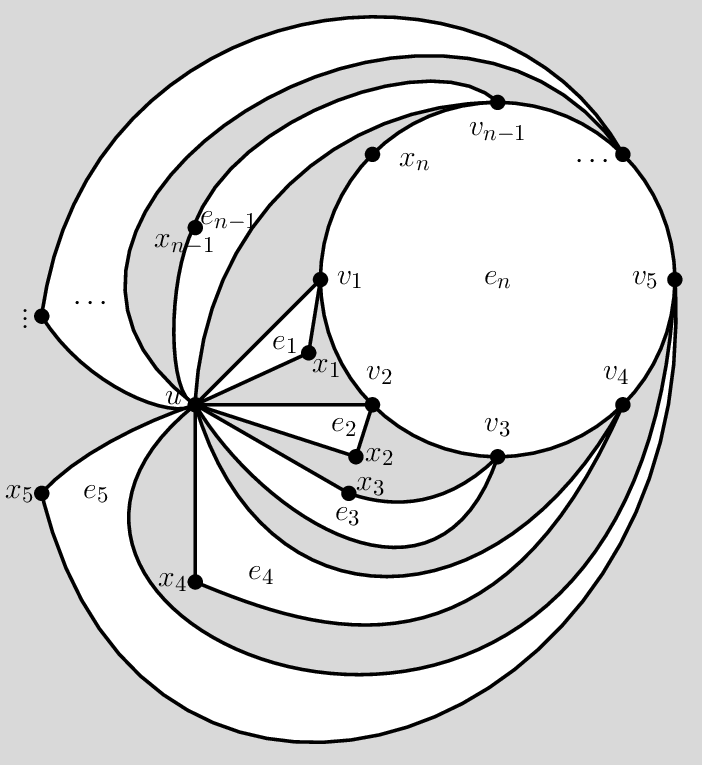}
\caption{A hypertree map with one cycle.}
    \label{p15}

\end{figure}
\begin{proof}

According to Theorem $\ref{th3}$, we can deduce that $\varepsilon(T^{A})=2[1-c(A)-c(A^{c})]+4n$, where $A\subseteq E(T)$. There are two distinct cases to consider.
\begin{itemize}
  \item [(1)] When \( e_{n} \notin A \), if \( e(A) = n - 1 \), then it follows that \( c(A) = 2 \) and \( c(A^{c}) = n + 1 \). Therefore, we have:
\[
c(A) + c(A^{c}) = n + 3.
\]
Conversely, if \( 0 \leq e(A) < n - 1 \), then we find that \( c(A) = 2n - 2e(A) \) and \( c(A^{c}) = e(A) + 1 \). Thus, in this case:
\[
c(A) + c(A^{c}) = 2n - e(A) + 1.
\]
  \item  [(2)] When \( e_{n} \in A \), if \( e(A) = 1 \), it is analogous to the case where \( e(A) = n - 1 \) as described in Case (1) above. Next, let us examine the case where \(2 \leq e(A) \leq n\).
Given that \(e(A) + e(A^{c}) = n\), we can derive that \(c(A) = n - e(A) + 1\) and \(c(A^{c}) = n - e(A^{c}) + [e(A) - 1] + 1 = 2e(A)\). Consequently, we find that \(c(A) + c(A^{c}) = n + e(A) + 1\).
\end{itemize}

In conclusion,
\[\varepsilon(T^{A})=\begin{cases}
2e(A)&\text{$e_{n}\notin A$ and $0\leq e(A)<n-1$},\\
2n-4&\text{$e_{n}\in A$ and $e(A)=1$ or $e_{n}\notin A$ and $e(A)=n-1$ },\\
2n-2e(A)&\text{$e_{n}\in A$ and $2\leq e(A)\leq n$}.
\end{cases}\]

Thus, we have
\[\partial_{\varepsilon_{T}}(z)=2+2z^{2n-4}+\sum\limits_{e_{n}\notin A,2\leq e(A)<n-1}z^{2e(A)}+\sum\limits_{e_{n}\in A,2\leq e(A)\leq n-1}z^{2n-2e(A)}\]
This can be expressed as  \[\partial_{\varepsilon_{T}}(z)=\begin{cases}
2+2z^{2n-4}+2\sum\limits_{i=1}^{[\frac{n}{2}]}\binom{n-1}{i}z^{2i}+2\sum\limits_{i=2}^{[\frac{n}{2}]}\binom{n-1}{i-1}z^{2n-2i}&\text{$n$ is odd},\\
2+2z^{2n-4}+2\sum\limits_{i=1}^{\frac{n}{2}-1}\binom{n-1}{i}z^{2i}+2\sum\limits_{i=2}^{\frac{n}{2}-1}\binom{n-1}{i-1}z^{2n-2i}+[\binom{n-1}
{\frac{n}{2}}+\binom{n-1}{\frac{n}{2}-1}]z^{n}&\text{$n$ is even}.
\end{cases}\]

\end{proof}

Let \( T_{n} \) be a \textit{4-uniform} hypertree map with \( E(T_{n}) = \{ e_{i} = (x_{2i-1}, x_{2i}, x_{2i+1}, x_{2i+2}) | i=1, 2, \ldots, n \} \), as illustrated in Figure \ref{p13}. Upon the removal of all vertices with degree 1 in \( T_{n} \), specifically \( x_{1}, x_{2}, x_{2n+1}, x_{2n+2} \), the resulting hypermap is denoted by \( H_{n}\) and is referred to as a hyper-ladder map. According to Theorem \ref{th8}, the partial-dual Euler-genus polynomial of \(H_{n} \) is equivalent to that of \(T_{n} \).
\begin{figure}[H]
    \centering

     \includegraphics[width=4in]{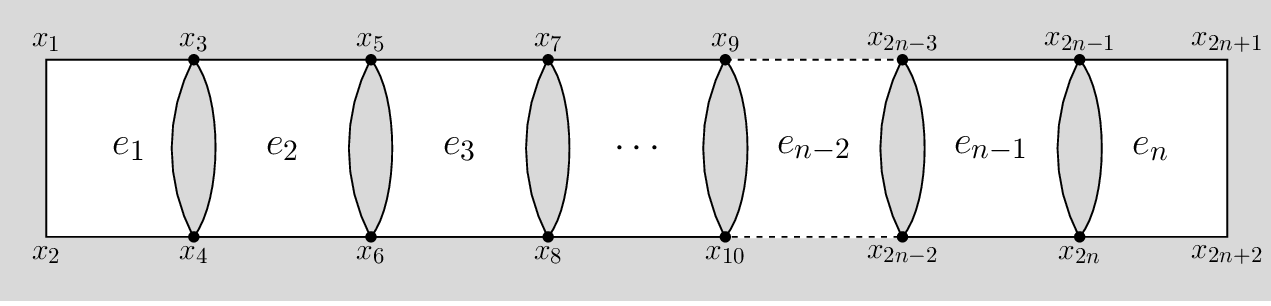}
     \includegraphics[width=4in]{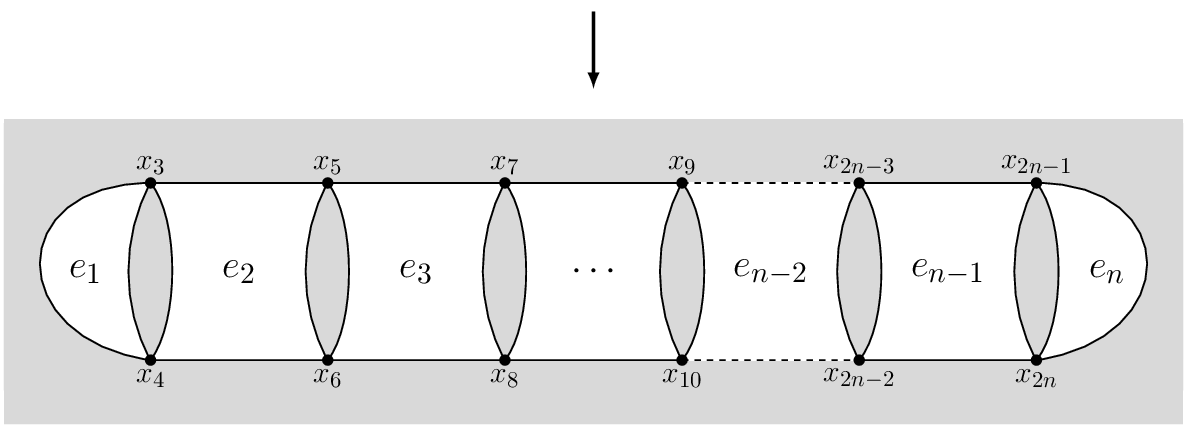}
    \caption{Two hypermaps $T_{n}$ (above) and  $H_{n}$ (below).}
    \label{p13}
\end{figure}

\begin{thm}\label{th5}
The partial-dual Euler-genus polynomial of the hyper-ladder map is $\partial_{\varepsilon_{H_{n}}}(z)=2(1+z^{2})^{n-1}$, where $n$ is a positive integer.
\end{thm}
\begin{proof}

For the sake of simplicity, we denote the hypermap consisting of a single edge \( e_n \) as \( e_n \). Since \( H_{n} = M_{H_{n-2}} \oplus_e e_{n} \), and given that \( \varepsilon[e_{n}^{\emptyset}] = \varepsilon[e_{n}^{\{e_{n}\}}] = 0 \), by Theorem \ref{th6}, we can conclude that:
\begin{equation}\label{eq1}
    \begin{aligned}
        \partial_{\varepsilon_{H_{n}}}&=2\sum\limits_{A\subseteq E(H_{n-2}\cup e_{n})}}z^{\varepsilon[H_{n-2}^{A\cap E(H_{n-2})}]
        +\varepsilon[e_{n}^{A\cap \{e_{n}\}}]+2(k_{1}+k_{2}-2)\\
                                                  &=2\sum\limits_{A\subseteq E(H_{n-2}\cup {e_{n}})}z^{\varepsilon[H_{n-2}^{A\cap E(H_{n-2})}]+2(k_{1}+k_{2}-2)}.
    \end{aligned}
\end{equation}

 If \( e_{n-2} \notin A \) and \( e_{n} \notin A \), then we have \( k_{1} = k_{2} = 1 \). If either \( e_{n-2} \in A \) or \( e_{n} \in A \), it follows that either \( k_{1} = 2, k_{2} = 1 \) or \( k_{1} = 1, k_{2} = 2\). If both \( e_{n-2} \in A \) and \( e_{n} \in A\), we conclude that \( k_{1}=k_{2}=2\).

According to Theorem \ref{th3}, we can deduce that \(\varepsilon(H_{n-2}^{A}) = \varepsilon(H_{n-2}^{A^{c}})\). Consequently, equation (\ref{eq1}) can be reformulated as follows:
\begin{equation}
    \begin{aligned}
    \partial_{\varepsilon_{H_{n}}}&=2\left(\sum\limits_{e_{n-2},e_{n}\notin A}z^{\varepsilon[H_{n-2}^{A\cap E(H_{n-2})}]}+2\sum\limits_{e_{n-2}\in A,e_{n}\notin A}z^{\varepsilon[H_{n-2}^{A\cap E(H_{n-2})}]+2}+\sum\limits_{e_{n-2},e_{n}\in A}z^{\varepsilon[H_{n-2}^{A\cap E(H_{n-2})}]+4}\right) \\
    \\
    &=\partial_{\varepsilon_{H_{n-2}}}(z)\cdot(1+z^{2})^{2}.\nonumber
    \end{aligned}
\end{equation}



Thus, we have
\[\partial_{\varepsilon_{H_{n}}}(z)=\begin{cases}
\partial_{\varepsilon_{H_{2}}}(z)(1+z^{2})^{2(\frac{n}{2}-1)}&\text{$n$ is even},\\
\partial_{\varepsilon_{H_{1}}}(z)(1+z^{2})^{2(\lceil\frac{n}{2}\rceil-1)}&\text{$n$ is odd}.
\end{cases}\]
Since $\partial_{\varepsilon_{H_{2}}}(z)=2(1+z^{2})$ and $\partial_{\varepsilon_{H_{1}}}(z)=2$,  the theorem follows.
\end{proof}


\vskip.51cm
\noindent Version: \printtime\quad\today\quad
\vskip.51cm
\end{document}